\documentclass[12pt]{amsart}

\pdfoutput=1

\usepackage[text={420pt,660pt},centering]{geometry}

\usepackage{color}
\usepackage{esint,amssymb}
\usepackage{graphicx,epsfig}
\usepackage{MnSymbol}
\usepackage{mathtools}
\usepackage[colorlinks=true, pdfstartview=FitV, linkcolor=blue, citecolor=blue, urlcolor=blue,pagebackref=false]{hyperref}
\usepackage{microtype}

\usepackage{bm}
\usepackage[font={footnotesize}]{caption}

\usepackage{mathrsfs}

\newtheorem{proposition}{Proposition}
\newtheorem{theorem}[proposition]{Theorem}
\newtheorem{lemma}[proposition]{Lemma}

\theoremstyle{remark}

\theoremstyle{definition}

\numberwithin{equation}{section}
\numberwithin{proposition}{section}

\newcommand{\Z}{\mathbb{Z}}
\newcommand{\N}{\mathbb{N}}
\newcommand{\R}{\mathbb{R}}
\renewcommand{\le}{\leqslant}
\renewcommand{\ge}{\geqslant}
\renewcommand{\leq}{\leqslant}
\renewcommand{\geq}{\geqslant}

\newcommand{\E}{\mathbb{E}}
\renewcommand{\P}{\mathbb{P}}
\newcommand{\F}{\mathcal{F}}
\newcommand{\Zd}{\mathbb{Z}^d}
\newcommand{\Rd}{{\mathbb{R}^d}}

\newcommand{\eps}{\varepsilon}

\renewcommand{\subset}{\subseteq}

\newcommand{\Ll}{\left}
\newcommand{\Rr}{\right}

\DeclareMathOperator{\dist}{dist}

\renewcommand{\bar}{\overline}
\renewcommand{\tilde}{\widetilde}
\renewcommand{\hat}{\widehat}

\newcommand{\un}{\underline}

\newcommand{\1}{\mathbf{1}}

\newcommand{\dr}{\partial}

\renewcommand{\a}{\mathbf{a}}
\newcommand{\ahom}{{\overbracket[1pt][-1pt]{\a}}}

\newcommand{\g}{\mathbf{g}}

\renewcommand{\S}{\mathbf{S}}

\renewcommand{\O}{\mathcal{O}}
\newcommand{\X}{\mathcal{X}}

\newcommand{\phil}{\phi^{(\lambda)}_{e_k}}
\newcommand{\Sl}{\mathbf{S}^{(\lambda)}}

\begin{document}

\title[Iterative method for rapidly oscillating coefficients]{An iterative method for elliptic problems with rapidly oscillating coefficients}

\begin{abstract}
We introduce a new iterative method for computing solutions of elliptic equations with random rapidly oscillating coefficients. Similarly to a multigrid method, each step of the iteration involves different computations meant to address different length scales. However, we use here the homogenized equation on all scales larger than a fixed multiple of the scale of oscillation of the coefficients. While the performance of standard multigrid methods degrades rapidly under the regime of large scale separation that we consider here, we show an explicit estimate on the contraction factor of our method which is independent of the size of the domain. We also present numerical experiments which confirm the effectiveness of the method, with openly available source code.
\end{abstract}

\author[S. Armstrong]{S. Armstrong}
\address[S. N. Armstrong]{Courant Institute of Mathematical Sciences, New York University, USA}
\email{scotta@cims.nyu.edu}

\author[A. Hannukainen]{A. Hannukainen}
\address[A. Hannukainen]{Department of Mathematics and Systems Analysis, Aalto University, Finland}
 \email{antti.hannukainen@aalto.fi}

\author[T. Kuusi]{T. Kuusi}
\address[T. Kuusi]{Department of Mathematics and Statistics, University of Helsinki, Finland}
 \email{tuomo.kuusi@helsinki.fi}

\author[J.-C. Mourrat]{J.-C. Mourrat}
\address[J.-C. Mourrat]{DMA, Ecole normale sup\'erieure,
CNRS, PSL Research University, Paris, France}
\email{mourrat@dma.ens.fr}

\date{\today}

\maketitle

%
%
%
%
%
%

\section{Introduction}

\subsection{Informal summary of results}

In this paper, we introduce a new iterative method for the numerical approximation of solutions of elliptic problems with rapidly oscillating coefficients. For definiteness, we consider the Dirichlet problem 
\begin{equation}  
\label{e.model.problem}
\left\{ 
\begin{aligned}
& -\nabla \cdot \Ll( \a(x) \nabla u \Rr) = f  & \mbox{in } &  U_r,\\
& u = w & \mbox{on } &  \partial U_r,
\end{aligned}
\right.
\end{equation}
where $r>0$ is the length scale of the problem, which is typically very large ($r\gg1$), and we write $U_r := r U$ where $U \subset \Rd$ is a bounded $C^{1,1}$ domain, in dimension $d\geq 2$.
The boundary condition $w$ belongs to $H^1(U_r)$, and the right-hand side $f$ belongs to $H^{-1}(U_r)$. The coefficients $\a(x)$ are symmetric, uniformly elliptic and H\"older continuous. Moreover, in order to ensure that \emph{quantitative homogenization} holds on large scales, we assume that the coefficients are sampled by a probability measure which is $\Zd$-stationary and has a unit range of dependence (see below for the precise formulation of these assumptions).
Our goal is to build a numerical method for the computation of $u$ which remains efficient in the regime of fast oscillations of the coefficient field  (which in our setting corresponds to the case in which the length scale is very large, $r \gg 1$) and does not rely on scale separation for convergence (the method computes the true solution for fixed~$r$ and not only in the limit $r\to \infty$).

\smallskip

In the absence of fast oscillations of the coefficient field, contemporary technology allows to access numerical approximations of elliptic problems involving billions of degrees of freedom. One of the most successful methods allowing to achieve such results is the \emph{multigrid method} (see \cite{FFT} for benchmarks). However, the performance of this method degrades as the coefficient field becomes more rapidly oscillating (see for instance \cite[Table~IV]{sundar2012parallel}).

\smallskip

We seek to remedy this problem by leveraging on \emph{homogenization}. While standard multigrid methods use a decomposition of the elliptic problem into a series of scales, the difficulty in our context is that the slow eigenmodes of the heterogeneous operator still have fast oscillations, and are thus not easily captured through a coarse representation. We overcome this by introducing a suitable variant of the multigrid method that succeeds in replacing the heterogeneous operator by the homogenized one 
on length scales larger than a large but finite multiple of the correlation length scale. The result is a new iterative method that converges exponentially fast in the number of iterations, each of which is relatively inexpensive to compute---the memory and number of computations required scale linearly in the volume, and the computation is very amenable to parallelization. We give a rigorous proof of convergence and present numerical experiments which establish the efficiency of the method from a practical point of view.  

\subsection{Statement of the main result}

We introduce some notation in order to state our main result. We begin with the precise assumptions on the coefficient field. 
We fix parameters~$\Lambda > 1$ and~$\alpha \in (0,1]$ and require our coefficient fields $\a(x)$ to satisfy
\begin{equation}
\label{e.holder}
\forall x,y \in \Rd, \quad |\a(y) - \a(x)| \le \Lambda|x-y|^\alpha
\end{equation}
and
\begin{equation}  
\label{e.unif.ell}
\forall x \in \Rd, \quad \forall \xi \in \Rd, \quad \Lambda^{-1} |\xi|^2 \le \xi \cdot \a(x) \xi \le \Lambda |\xi|^2.
\end{equation}
We denote by~$\R^{d\times d}_{\mathrm{sym}}$ the set of~$d$-by-$d$ real symmetric matrices and define
\begin{equation*}  
\Omega := \Ll\{ \a : \R^d \to \R^{d \times d}_\mathrm{sym} \ \mbox{satisfying \eqref{e.holder} and \eqref{e.unif.ell}} \Rr\}.
\end{equation*}
For each Borel set $V \subset \Rd$, we denote by $\F_V$ the Borel $\sigma$-algebra on $\Omega$ generated by the family of mappings
\begin{equation*}  
\a \mapsto \int_\Rd \chi \, \a_{ij} , \qquad i,j \in \{1,\ldots,d\}, \ \chi \in C^\infty_c(V).
\end{equation*}
We also set $\F := \F_{\Rd}$.
For each $y \in \Rd$, we denote by $T_y:\Omega\to\Omega$ the action of translation by~$y$:
\begin{equation*}  
\forall x \in \Rd, \quad T_y \a(x) := \a(x+y).
\end{equation*}
We assume that~$\P$ is a probability measure on $(\Omega, \F)$ satisfying:
\begin{itemize}  
\item stationarity with respect to $\Zd$-translations: for every $y \in \Zd$ and $A \in \F$,
\begin{equation*}  
\P \Ll[ T_y A \Rr] = \P \Ll[ A \Rr] ;
\end{equation*}
\item unit range of dependence: for every Borel sets $V,W \subset \Rd$,
\begin{equation*}  
\dist(V,W) \ge 1 \implies \F_{V} \text{ and } \F_{W} \mbox{ are $\P$-independent.}
\end{equation*}
\end{itemize}
The expectation associated with the probability measure $\P$ is denoted by $\E$. We recall that, by classical homogenization theory (see~\cite{PV1,AKMbook}), the heterogeneous operator $-\nabla \cdot \a(x) \nabla$ homogenizes to the homogeneous operator $-\nabla \cdot \ahom \nabla$, where $\ahom \in \R^{d\times d}$ is a deterministic, constant, positive definite matrix. For every $s, \theta > 0$ and random variable $X$, we write
\begin{equation}
\label{e.def.Os}
X \le \O_s \Ll( \theta \Rr) \quad \text{if and only if} \quad  \E \Ll[  \exp \Ll( \Ll( \theta^{-1} \max(X,0) \Rr)^s \Rr)\Rr]  \le 2.
\end{equation}
We also set, for every $\lambda \in (0,1]$,
\begin{equation}
\label{e.def.llambda}
\ell(\lambda) := 
\Ll\{
\begin{aligned}
& \Ll( \log (1+\lambda^{-1}) \Rr)^\frac 1 2  & \quad \text{if } d = 2, \\
& 1 & \quad \text{if } d \ge 3.
\end{aligned}
\Rr.
\end{equation}
For notational convenience, from now on we will suppress the explicit dependence on the spatial variable in the operator $-\nabla \cdot \a(x) \nabla$ and simply write $-\nabla \cdot \a \nabla$. 
\smallskip

We now state the main result of the paper. We recall that $\P$ is a probability measure on $(\Omega,\F)$ which specifies the law of the coefficient field~$\a(x)$ and satisfies the assumptions stated above, that~$\ahom$ is the homogenized matrix associated to~$\P$, and that $U\subseteq \Rd$ is a bounded domain with $C^{1,1}$ boundary.

\begin{theorem}[$H^1$ contraction]
\label{t.contract}
For each~$s \in (0,2)$, there exists a constant $C(s,U,\Lambda,\alpha,d) < \infty$ such that the following statement holds.
Fix $r \geq 1$, $\lambda \in \left[r^{-1},1\right]$, $f\in H^{-1}(U_r)$, $w\in H^1(U_r)$ and let $u \in w+H^1_0(U_r)$ be the solution of~\eqref{e.model.problem}. Also fix a function~$v \in w + H^1_0(U_r)$ and define the functions $u_0, \bar u, \tilde u \in H^1_0(U_r)$ to be the solutions of the following equations (with null Dirichlet boundary condition on~$\partial U_r$):
\begin{align*}  
(\lambda^2 - \nabla \cdot \a \nabla)u_0 & 
= f +  \nabla \cdot \a \nabla v &  \text{in } U_r,
\\
- \nabla \cdot \ahom \nabla\bar u & = \lambda^2 u_0 &  \text{in } U_r, 
\\
(\lambda^2 - \nabla \cdot \a \nabla)\tilde u & = (\lambda^2-\nabla \cdot \ahom \nabla )\bar u&  \text{in } U_r.
\end{align*}
For~$\hat v \in  w+H^1_0(U_r)$ defined by
\begin{equation}
\label{e.def.hatv}
\hat v := v + u_0  + \tilde u,
\end{equation}
we have the estimate
\begin{equation}  
\label{e.contract}
\left\| \nabla ( \hat v-u ) \right\|_{L^2(U_r)} \le \O_s \Ll(C \ell(\lambda)^\frac 1 2 \, \lambda^{\frac 1 2} \, \left\| \nabla ( v-u )  \right\|_{L^2(U_r)} \Rr).
\end{equation}
\end{theorem}

The function $u \in H^1(U_r)$ appearing in Theorem~\ref{t.contract} is the unknown we wish to approximate, and $v \in H^1(U_r)$ should be thought of as the current approximation to~$u$. The function~$\hat{v}$ is then the new, updated approximation to~$u$ and the estimate~\eqref{e.contract} says that, if $\lambda$ is chosen small enough, then the error in our approximation will be reduced by a multiplicative factor of~$1/2$. 
As explained more precisely around \eqref{e.unif.contract} below, we can then iterate this procedure and obtain rapid convergence to the solution. The only assumption we make on~$v$ is that it satisfies the correct boundary condition, that is, $v \in w + H^1_0(U_r)$. In particular, we may  begin the iteration with~$v = w$ as the initial guess (or any other function with the correct boundary condition).
The computation of $\hat v$ reduces to solving the problems for $u_0, \bar u$, and $\tilde u$ listed in the  statement, and the point is that each of these problems is relatively inexpensive to compute, provided that~$\lambda$ is not too small. 
A fundamental aspect of the result is therefore that the required smallness of the parameter~$\lambda$ (so that~\eqref{e.contract} gives us a strict contraction in~$H^1$) does not depend on the length scale~$r$ of the problem. In other words, we may need to take~$\lambda$ to be small, but it will still be of order one, no matter how large~$r$ is. 

\smallskip

Similarly to standard multigrid methods, the equation for $u_0$ is meant to resolve the small-scale discrepancies between $u$ and $v$. Note that the equation for $u_0$ can be rewritten as
\begin{equation*}  
(\lambda^2 - \nabla \cdot \a \nabla)u_0  = - \nabla \cdot \a \nabla (u-v) \qquad \text{in } U_r.
\end{equation*}
The parameter~$\lambda^{-1}$ is the characteristic length scale of this problem, and in practice we will take it to be some fixed multiple of the scale of oscillations of the coefficients. The computation of $u_0$ can thus be decomposed into a large number of essentially unrelated elliptic problems posed on subdomains of side length of the order of $\lambda^{-1}$. In analogy with multigrid methods, we may also think of~$\lambda^{-2}$ as the number of elementary pre-smoothing steps performed during one global iteration. 

\smallskip

As announced, we then use the homogenized operator on scales larger than~$\lambda^{-1}$. This is what the problem for $\bar u$ is meant to capture. Since the elliptic problem for~$\bar u$ involves the homogenized operator $-\nabla \cdot \ahom \nabla$, it can be solved efficiently using the standard multigrid method. We note that the equation for~$\bar u$ can be rewritten, if desired, in the form
\begin{equation}  
\label{e.alt.eq.baru}
- \nabla \cdot \ahom \nabla\bar u  = -\nabla \cdot \a \nabla(u- v-u_0) \qquad  \text{in } U_r.
\end{equation}

\smallskip

The final step of the iteration, involving the definition of $\tilde u$, is meant to add back some small-scale details to the function $\bar u$. It is analogous to the post-smoothing step in the standard $V$-cycle implementation of the multrigrid method, and the parameter $\lambda^{-2}$ represents the number of post-smoothing steps.

\smallskip

We next discuss the more probabilistic aspects involved in the statement of Theorem~\ref{t.contract}. Since the coefficient field is random, the statement of this theorem can only be valid with high probability, but not almost surely. Indeed, with non-zero probability, the coefficient field can be essentially arbitrary, and on such small-probability events, the idea of leveraging on homogenization can only perform badly (recall that we aim for a convergence result for large but~fixed $r$, as opposed to asymptotic convergence). It may help the intuition to observe that, by Chebyshev's inequality, the assumption of \eqref{e.def.Os} implies that
\begin{equation}  
\label{e.tail}
\forall x \ge 0, \quad \P \Ll[ X \ge \theta x \Rr] \le 2 \exp(-x^s),
\end{equation}
and that conversely, the assumption of \eqref{e.tail} implies that $X \le \O_s(C \theta)$ for some constant $C(s) < \infty$ (see \cite[Lemma~A.1]{AKMbook}).

\smallskip

We remark that Theorem~\ref{t.contract} is new even when restricted to the subclass of periodic coefficient fields. In this case, both the probabilistic part of the estimate as well as the  logarithmic factor of $\ell(\lambda)$ are not present, and~\eqref{e.contract} can be replaced with the simpler form
\begin{equation*}  
\|\nabla(\hat v-u)\|_{L^2(U_r)} \le C 
\lambda^{\frac 1 2} \, \|\nabla(v - u)\|_{L^2(U_r)}.
\end{equation*}

\smallskip

We stress that the probabilistic statement in \eqref{e.contract} is valid for each fixed choice of~$u,v \in H^1(U_r)$. In fact, further work inspired by the first version of the present paper allowed to obtain the following stronger, uniform estimate \cite{chenlin}. For each $s \in (0,2)$, there exist a constant $C(s,p,U,\Lambda,d) < \infty$ and, for each $r \geq 1$ and $\lambda \in [r^{-1},1]$, a random variable $\X_{s,r,\lambda} : \Omega \to [0,+\infty]$  satisfying 
\begin{equation*}  
\X_{s,r,\lambda} \le \O_s \Ll( C\Rr) 
\end{equation*}
such that, for every $u,v \in H^1(U_r)$ and $\hat v$ as in the statement of Theorem~\ref{t.contract}, 
\begin{equation}  
\label{e.unif.contract}
\left\| \nabla (\hat v - u) \right\|_{ L^2(U_r)} 
\leq 
\X_{s,r,\lambda} \,
\ell(\lambda)^{\frac12} \,
\lambda^\frac 1 2 \, 
(\log r)^\frac 1 s \, 
\left\|\nabla (v-u)\right\|_{ L^2(U_r)}.
\end{equation}
Moreover, the proof given in \cite{chenlin} does not require that the coefficient field be H\"older continuous. As is apparent in \eqref{e.unif.contract}, the price one has to pay for the uniformity of this estimate in the functions $u$ and $v$ is a slight degradation of the contraction factor, by a slowly diverging logarithmic factor of the domain size. Due to  randomness, uniform estimates such as \eqref{e.unif.contract} must necessarily contain some logarithmic divergence in the domain size. Indeed, consider for instance the case of a coefficient field given by a random checkerboard in which we toss a fair coin, independently for each $z \in \Z^d$, the coefficient field in $z + [0,1)^d$ to be either $I_d$ or $2I_d$. Then, with probability tending to one as $r$ tends to infinity, there will be in the domain $U_r$ a region of space of side length of the order of $(\log r)^\frac 1 d$ where the coefficient field is constant equal to $I_d$. If the support of the solution we seek is concentrated in this region, then the iteration described in Theorem~\ref{t.contract} will perform badly unless $\lambda^{-1}$ is chosen larger than $(\log r)^\frac 1 d$. 

\smallskip

The iteration proposed in Theorem~\ref{t.contract} requires the user to make a judicious choice of the length scale~$\lambda^{-1}$. Ideally, it would be preferable to devise an adaptive method which discovers a good choice for~$\lambda^{-1}$ automatically. The contraction of the iteration would then be guaranteed with probability one, and more subtle probabilistic quantifiers would instead enter into the complexity analysis of the method.
A suitably designed adaptive algorithm would likely 
also work on more general coefficient fields than those considered here, allowing for instance to drop the assumption of stationarity. An assumption of approximate local stationarity would then also enter into the complexity analysis of the method. We leave the  development of such adaptive methods to future work. 

\smallskip

The method proposed here also requires that the user computes~$\ahom$ beforehand. An efficient method for doing so was presented in \cite{efficient,ahom}; see also~\cite{Gloria, cemracs} and references therein for previous work on this problem. Moreover, one can check that in order to guarantee the contraction property of the iteration described in Theorem~\ref{t.contract}, say by a factor of~$1/2$, a coarse approximation of~$\ahom$, which may be off by a small but fixed positive amount, suffices.

\smallskip

The proof of Theorem~\ref{t.contract} can be modified so that the~$L^2$ norms in~\eqref{e.contract} are replaced by $L^p$ norms, for any exponent~$p<\infty$. Up to some additional logarithmic factors in~$\lambda$, the contraction factor in the estimate would then be of order~$\lambda^{\frac 1p}$ rather than $\lambda^{\frac 12}$. This modification requires the application of large-scale Calder\'on-Zygmund-type~$L^p$ estimates which can be found in~\cite[Chapter 7]{AKMbook}. The main required modification to the proof of Theorem~\ref{t.contract} is simply to upgrade the two-scale expansion result of Theorem~\ref{t.twoscale} from $p=2$ to larger exponents by adapting the argument of~\cite[Theorem 7.10]{AKMbook}.

\subsection{Previous works}

There has been a lot of work on numerical algorithms that become sharp only in the limit of infinite scale separation (see for instance~ \cite{NJW, MS, brandt, eq-free,eh-book,E-book,  hmm} and the references therein). That is, the error between the true solution~$u$ and its numerical approximation becomes small only as~$r\to \infty$. Such algorithms typically have a computational complexity scaling sublinearly with the volume of the domain. An example of such a method in the context of the homogenization problem considered here is to compute an approximation of the solution to the homogenized equation. In addition to relying on scale separation, we note that such a sublinear method can only give an accurate global approximation in a weaker space such as~$L^2$, but not in stronger norms such as~$H^1$ which are sensitive to small scale oscillations. 

\smallskip

We now turn our attention to numerical algorithms that, like ours, converge to the true solution for each finite value of~$r$. 
As pointed out in \cite{knap1,knap2}, direct applications of standard multigrid methods result in coarse-scale systems that do not capture the relevant large-scale properties of the problem. Indeed, standard  coarsening procedures produce effective coefficients that are simple arithmetic averages of the original coefficient field, instead of the homogenized coefficients. To remedy this problem,  \cite{knap1,knap2} propose more subtle, matrix-dependent choices for the restriction and prolongation operators. The idea is to try to approximate a Schur complement calculation, while preserving some calculability constraints such as matrix sparsity. The method proposed there is shown numerically to perform better than simple averaging, but no theoretical guarantee is provided.

\smallskip

In \cite{EL1,EL2}, the authors propose, in the periodic setting, to solve local problems for the correctors, deduce locally homogenized coefficients, and build coarsened operators from these. For the special two-dimensional case with $\a(x) = \tilde \a(x_1 - x_2)$ for some $1$-periodic $\tilde \a \in C([0,1];\R^{2\times 2}_{\mathrm{sym}})$, they show (in our notation) that $O(r^{\frac 5 3} \log r)$ smoothing steps suffice to guarantee the contractivity of the two-step multigrid method (assuming that the chosen coarsening scale is a bounded multiple of the oscillation scale). For comparison, this roughly corresponds to the choice of $\lambda \simeq r^{-\frac 5 6}$ in our method. They also report better numerical performance than predicted by their theoretical arguments.

\smallskip

Beyond our current assumption of stationarity of the coefficient field, one can look for numerical methods for the resolution of general elliptic problems with rapidly oscillating coefficients. Possibly the simplest such method is to rely on the uniform ellipticity assumption~\eqref{e.unif.ell} and appeal to a preconditioned conjugate gradient method, using the standard Laplacian as a preconditioner. However, the norm that is contracted at each iteration of this algorithm is the~$L^2$ norm, as opposed to a contraction of the $H^1$ norm as obtained in the present paper\footnote{Naturally, the $L^2$ and $H^1$ norms become equivalent after discretization; but a theoretical guarantee of contraction in $H^1$ ensures that the high frequencies can be resolved efficiently after only a few steps of the iteration, irrespectively of the size of the mesh refinement.}. Moreover, the performance of this method degrades quickly if the ellipticity ratio $\Lambda$ becomes large. In contrast, using some of the results and techniques of \cite{AD,D1}, generalizations of Theorem~\ref{t.contract} have now been obtained in the highly degenerate case of perforated media of percolation type, for which  $\Lambda = \infty$; see \cite{chenlin2}.

\smallskip

Algebraic multigrid methods are intended to solve completely arbitrary linear systems of equations, by automatically discovering a hierarchy of coarsened problems~\cite{amg}. In practice, it is however necessary to make some judicious choices of coarsening operators. In a sense, the present contribution as well as those of \cite{knap1,knap2,EL1,EL2} are descriptions of specific coarsening procedures which, under stronger assumptions such as stationarity, are shown to have fast convergence properties.

\smallskip

Many alternative approaches to the computation of elliptic problems with arbitrary coefficient fields have been developed. We mention in particular, without going into details, hierarchical matrices \cite{bebendorf}, generalized multiscale finite element methods \cite{BL,EGH,GGS}, polyharmonic splines \cite{OZB}, local orthogonal decompositions \cite{malpet}, subspace correction methods \cite{KY} and gamblets \cite{roulette}. While methods such as gamblets have been shown theoretically to have essentially linear complexity under weaker assumptions than those explored in the present paper, the construction and storage of the hierarchy of gamblets may actually be quite expensive in practice (we are not aware of large-scale computations that use gamblets; the main numerical example in~\cite{roulette} has $2^{24} \simeq 1.7 \cdot 10^7$ degrees of freedom). Methods such as local orthogonal decompositions introduce an intermediate scale, often denoted by~$H$, inbetween the microscopic and the macroscopic scales, and an adapted basis of local functions is computed at this level. In this framework, the numerical error is bounded from below by a multiple of $H$. The method presented in Theorem~\ref{t.contract} shares some aspects of this idea in that it also introduces an intermediate scale $\lambda^{-1}$; however, the final numerical error is not constrained by this choice, and can be made arbitrarily low irrespectively of the value of~$\lambda$.

\smallskip

The very recent work \cite{fgp} probably comes closest to the goals of the present work. Under assumptions similar to ours, they analyze the performance of the method of local orthogonal decompositions (LOD). We believe that the method presented here has fundamental advantages over LOD. One aspect is that, as explained above, the error in LOD is bounded from below by a multiple of the intermediate scale (often denoted by $H$), while this is not so in our approach (in which the intermediate scale is $\lambda^{-1}$). Moreover, our result guarantees an approximation of the true solution in~$H^1$, while the results of \cite{fgp} only provide with $L^2$ estimates. Finally, our method is very easy to implement, as can be verified by the curious reader since the source code of our numerical tests is openly available, see \eqref{e.github}; on the other hand, the finite but possibly large overlaps of the local orthogonal frame might in practice cause significant computational overheads.

\smallskip

\subsection{Organization of the paper}
We introduce some more notation in Section~\ref{s.notation}. Section~\ref{s.proof} is devoted to the proof of Theorem~\ref{t.contract}. We report on our numerical results in Section~\ref{s.numerics}. Finally, an appendix recalls some classical Sobolev and elliptic estimates for the reader's convenience.

%
%
%
%
%
%

\section{Notation} 
\label{s.notation}

In this section, we collect some notation used throughout the paper. Recall that the notation~$\O_s(\cdot)$ was defined in \eqref{e.def.Os}. We will need the following fact, which says that~$\O_s$ is behaving like a norm: for each $s \in (0,\infty)$, there exists~$C_s < \infty$ (with $C_s = 1$ for $s \ge 1$) such that the following triangle inequality for $\O_s(\cdot)$ holds: for any measure space $(E,\mathcal{S},\mu)$, measurable function $\theta: E \to (0,\infty)$ and jointly measurable family $\left\{ X(z) \right\}_{z \in E}$ of random variables, we have (see \cite[Lemma~A.4]{AKMbook})
\begin{equation} 
\label{e.Osums}
\forall z\in E, \   X(z) \le \O_s(\theta(z))
\implies
\int_{E} X\,d\mu \leq \O_s\left( C_s \int_E \theta \,d\mu \right). 
\end{equation}
We denote by $(e_1,\ldots,e_d)$ the canonical basis of $\Rd$, and write $B(x,r) \subset \Rd$ for the Euclidean ball centered at $x \in \Rd$ and of radius $r > 0$. For a Borel set $V \subset \Rd$, we denote its Lebesgue measure by $|V|$. If $|V| < \infty$, then for every $p \in [1,\infty)$ and $f \in L^p(V)$ we write the scaled~$L^p$ norm of~$f$ by 
\begin{equation*}  
\|f\|_{\un L^p(V)} := \Ll( |V|^{-1} \int_V f^p \Rr)^{\frac 1 p} = |V|^{-\frac 1 p} \|f\|_{L^p(V)}.
\end{equation*}
For each $k \in \N$, we denote by $H^{k}(V)$ the classical Sobolev space on $V$, whose norm is given by
\begin{equation*}  
\|f\|_{H^{k}(V)} := \sum_{0 \le |\beta| \le k} \|\partial^\beta f\|_{L^2(V)}.
\end{equation*}
In the expression above, the parameter $\beta = (\beta_1, \ldots, \beta_d)$ is a multi-index in $\N^d$, and we used the notation
\begin{equation*}  
|\beta| := \sum_{i = 1}^d \beta_i \quad \text{ and } \quad \partial^\beta f = \partial_{x_1}^{\beta_1} \cdots \partial_{x_d}^{\beta_d} f.
\end{equation*}
Whenever $|V| < \infty$, we define the scaled Sobolev norm by
\begin{equation*}  
\|f\|_{\un H^{k}(V)} := \sum_{0 \le |\beta| \le k} |V|^{\frac{|\beta| - k}{d}}\|\partial^\beta f\|_{\un L^2(V)}.
\end{equation*}
We denote by $H^1_0(V)$ the completion in $H^1(V)$ of the space $C^\infty_c(V)$ of smooth functions with compact support in $V$.
We write $H^{-1}(V)$ for the dual space to $H^1_0(V)$, which we endow with the (scaled) norm
\begin{equation*}  
\|f\|_{\un H^{-1}(V)} := \sup \Ll\{ |V|^{-1} \int_V f \, g, \quad g \in H^1_0(V), \ \|g\|_{\un H^1(V)} \le 1 \Rr\} .
\end{equation*}
The integral sign above is an abuse of notation and should be understood as the duality pairing between $H^{-1}(V)$ and $H^1_0(V)$. The spaces $H^{-1}(V)$ and $H^1_0(V)$ can be continuously embedded into the space of distributions, and we make sure that the duality pairing is consistent with the integral expression above whenever $f$ and $g$ are smooth functions.
For every~$r > 0$ and~$x \in \Rd$, we denote the time-slice of the heat kernel which has length scale ~$r$ by
\begin{equation}
\label{e.gaussian}
\Phi_r(x) := (4\pi r^2)^{-\frac d 2} \exp \Ll( -\frac{x^2}{4r^2} \Rr) .
\end{equation}
We denote by $\zeta \in C^\infty_c(\Rd)$ the standard mollifier
\begin{equation}
\label{e.def.zeta}
\zeta(x) : = \left\{ 
\begin{aligned}
& c_d \, \exp\left( - (1-|x|^2)^{-1} \right) & \mbox{if} & \ |x|<1,\\
& 0 & \mbox{if} & \ |x| \geq 1,
\end{aligned}
\right.
\end{equation}
where the constant~$c_d$ is chosen so that $\int_{\R^d} \zeta = 1$. 
For $f \in L^p(\Rd)$ and $g \in L^{p'}(\Rd)$ with $\frac 1 p + \frac 1 {p'} = 1$, we denote the convolution of $f$ and $g$ by
\begin{equation*}  
f \ast g (x) := \int_\Rd f(y) g(x-y) \, dy.
\end{equation*}

%
%
%
%
%
%
\section{Proof of Theorem~\ref{t.contract}}
\label{s.proof}

This section is devoted to the proof of~Theorem~\ref{t.contract}. We begin by introducing the notion of \emph{(first-order) corrector}: for each $p \in \Rd$, the corrector in the direction of $p$ is the function $\phi_p \in H^1_{\mathrm{loc}}(\Rd)$ solving
\begin{equation*}  
-\nabla \cdot \a \Ll( p + \nabla \phi_p \Rr) = 0 \qquad \text{in } \Rd,
\end{equation*}
and such that the mapping $x \mapsto \nabla \phi_p(x)$ is $\Z^d$-stationary and satisfies
\begin{equation*}  
\E \Ll[ \int_{[0,1]^d} \nabla \phi_p \Rr] = 0.
\end{equation*}
The corrector $\phi_p$ is unique up to an additive constant (see~\cite[Definition~4.2]{AKMbook} for instance). We also recall that one can define the homogenized matrix $\ahom \in \R^{d\times d}_\mathrm{sym}$ via the formula
\begin{equation*}  
\forall p \in \Rd, \quad \ahom p = \E \Ll[ \int_{[0,1]^d} \a(p + \nabla \phi_p )\Rr] ,
\end{equation*}
or equivalently,
\begin{equation*}  
\forall p \in \Rd, \quad p \cdot \ahom p = \E \Ll[ \int_{[0,1]^d} (p + \nabla \phi_p ) \cdot \a (p + \nabla \phi_p )\Rr] ,
\end{equation*}
and in particular, as a consequence of \eqref{e.unif.ell}, we have
\begin{equation}
\label{e.ahom.ell}
\forall \xi \in \Rd, \quad \Lambda^{-1} |\xi|^2 \le \xi \cdot \ahom \xi \le \Lambda |\xi|^2.
\end{equation}
For each $k \in \{1,\ldots,d\}$ and $\lambda > 0$, we denote 
\begin{equation*}  
\phil := \phi_{e_k} - \phi_{e_k} \ast \Phi_{\lambda^{-1}}.
\end{equation*}

A key ingredient in the proof of Theorem~\ref{t.contract} is the following \emph{quantitative two-scale expansion} for the operator~$\left(\lambda^2 - \nabla \cdot \a\nabla \right)$. It is the only input from the quantitative theory of stochastic homogenization used in this paper and it follows from some estimates which can be found in~\cite{AKMbook}.

\begin{theorem}[Two-scale expansion and error estimate]
\label{t.twoscale}
For each $s \in (0,2)$, there exists a constant $C(s,U,\Lambda,\alpha, d) < \infty$ such that, for every $r \ge 1$, $\lambda \in [r^{-1},1]$, and $\bar v \in H^1_0(U_r) \cap H^{2}(U_r)$, defining
\begin{equation}  
\label{e.def.w}
w := \bar v  + \sum_{k = 1}^d \phil \, \partial_{x_k} \bar v,
\end{equation}
we have the estimate
\begin{equation}  
\label{e.twoscale}
 \Ll\| \nabla \cdot \Ll( \a \nabla w -  \ahom \nabla \bar v \Rr)\Rr\|_{\un H^{-1}(U_r)} 
\le 
\O_s  \Ll( C \ell(\lambda)\|\bar v\|_{\un H^{2}(U_r)} +  C \lambda^{\frac d 2} \, \|\bar v\|_{\un H^{1}(U_r)} \Rr).
\end{equation}
Moreover, 
for every $\mu \in [0,\lambda]$ and
$v \in H^1_0(U_r)$ such that 
\begin{equation}  
\label{e.eq.v.barv}
\Ll(\mu^2 -\nabla \cdot \a \nabla\Rr) v = \Ll(\mu^2 - \nabla \cdot \ahom \nabla\Rr) \bar v,
\end{equation}
we have the estimate
\begin{multline}  
\label{e.error.estimates}
\Ll\| v - w \Rr\|_{\un H^1(U_r)} + \Ll(\mu + r^{-1} \Rr) \|v-\bar v\|_{\un L^2(U_r)} + \Ll(\mu + r^{-1} \Rr)^2 \|v-\bar v\|_{\un H^{-1}(U_r)} 
\\
\le \O_s  \Ll( C \Ll(\mu \ell(\lambda) + \lambda^\frac d 2\Rr) \|\bar v\|_{\un H^1(U_r)} +  C \ell(\lambda)^\frac 1 2 \,\| \bar v \|_{\un H^1(U_r)}^\frac 1 2 \, \|\bar v\|_{\un H^2(U_r)}^\frac 1 2 +  C \ell(\lambda) \|\bar v\|_{\un H^2(U_r)} \Rr).
\end{multline}
\end{theorem}

\smallskip

The proof of Theorem~\ref{t.twoscale} follows that of a similar result from Chapter~6 of~\cite{AKMbook}. The main difference here is the presence of the zeroth order term with the factor of~$\mu^2$, which presents no additional difficulty. We begin by recalling the concept of a flux corrector and stating some estimates on the correctors proved in~\cite{AKMbook}. 

\smallskip

For each $p \in \Rd$, we denote the (centered) flux of the corrector $\phi_p$ by
\begin{equation*}  
(\g_{p,i})_{1 \le i \le d} = \g_p := \a(p + \nabla \phi_p) - \ahom p.
\end{equation*}
Since $\nabla \cdot \g_p = 0$, the flux of the corrector admits a representation as the ``curl'' of some vector potential, by Helmholtz's theorem. This vector potential, the \emph{flux corrector}, will be useful for the proof of Theorem~\ref{t.twoscale}. For each $p \in \Rd$, the vector potential $(\S_{p,ij})_{1 \le i,j \le d}$ is a matrix-valued random field with entries in $H^1_{\mathrm{loc}}(\Rd)$ satisfying, for each $i,j \in \{1,\ldots,d\}$,
\begin{equation*}  
\S_{p,ij} = - \S_{p,ji},
\end{equation*}
\begin{equation}  
\label{e.def.S}
\nabla \cdot \S_{p} = \g_p,
\end{equation}
and such that $x \mapsto \nabla \S_{p,ij}(x)$ is a stationary random field with mean zero. In \eqref{e.def.S}, we used the shorthand notation
\begin{equation*}  
(\nabla \cdot \S_e)_i := \sum_{j = 1}^d \partial_{x_j} \S_{e,ij}. 
\end{equation*}
The conditions above do not specify the flux corrector uniquely. One way to ``fix the gauge'' is to enforce that, for each $i,j \in \{1,\ldots,d\}$,
\begin{equation*}  
\Delta \S_{p,ij} = \partial_{x_j} \g_{p,i} - \partial_{x_i} \g_{p,j}.
\end{equation*}
This latter choice then defines $\S_{p,ij}$ uniquely, up to the addition of a constant. We refer to \cite[Section~6.1]{AKMbook} for more precision on this construction.
We set
\begin{equation}  
\label{e.Sl}
\Sl_{e} := \S_e - \S_e \ast \Phi_{\lambda^{-1}}.
\end{equation}
The fundamental ingredient for the proof of Theorem~\ref{t.twoscale} is the following proposition, which quantifies the convergence to zero of the spatial averages of the gradients of the correctors.
\begin{proposition}[Corrector estimates]
For each $s \in (0,2)$, there exists a constant $C(s,U,\Lambda,\alpha,d) < \infty$ such that for every $\lambda \in (0,1)$, $x \in \Rd$ and $i,j,k \in \{1,\ldots,d\}$,
\label{p.corrector}
\begin{equation}
\label{e.apriori.nablaphi}
|\nabla \phi_{e_k}(x)| \le \O_s \Ll( C \Rr) ,
\end{equation}
\begin{equation}
\label{e.bound.gradphi}
\Ll| \Ll(\nabla \phi_{e_k} \ast \Phi_{\lambda^{-1}}\Rr) (x) \Rr| 
+ \Ll| \Ll(\nabla \S_{e_k,ij} \ast \Phi_{\lambda^{-1}} \Rr)(x) \Rr| \le \O_s \Ll( C \lambda^{\frac d 2} \Rr) ,
\end{equation}
\begin{equation}
\label{e.bound.phi}
\Ll|\phil(x)\Rr| + \Ll|\Sl_{e_k,ij}(x)\Rr| = \O_s \Ll( C \ell(\lambda) \Rr) .
\end{equation}
\begin{proof}
By \cite[Lemma~4.4]{AKMbook}, we have
\begin{equation*}  
\|\nabla \phi_{e_k}\|_{L^2(B(0,1))} \le \O_s \Ll( C \Rr) .
\end{equation*}
By the assumption of \eqref{e.holder}, we can apply standard Schauder estimates, see e.g.\ \cite[Theorems 3.1 and 3.8]{HL}, to deduce \eqref{e.apriori.nablaphi}. The estimates in \eqref{e.bound.gradphi} are proved in \cite[Theorem~4.9 and Proposition~6.2]{AKMbook}. The estimates in \eqref{e.bound.phi} also follow from \cite[Theorem~4.9 and Proposition~6.2]{AKMbook}, combined with the assumption of \eqref{e.holder} and the Schauder estimate in \cite[Corollary~3.2 and Theorem~3.8]{HL}. 
\end{proof}
\end{proposition}

In the next lemma, we provide a convenient representation of~$\nabla \cdot \a \nabla w$ in terms of the correctors.

\begin{lemma}
\label{l.two-scale}
Let $\lambda >0$, $\bar v \in H^1(U_r)$, and let $w \in H^1(U_r)$ be defined by~\eqref{e.def.w}. 
Then
\begin{equation*}  
\nabla \cdot \Ll(\a \nabla w -  \ahom \nabla \bar v\Rr) = \nabla \cdot \mathbf F,
\end{equation*}
where the $i$-th component of the vector field $\mathbf F$ is given by
\begin{multline}  
\label{e.def.Fi}
\mathbf F_i := 
\sum_{j,k = 1}^d  \Ll(\a_{ij} \phil -  \Sl_{e_k,ij}  \Rr)  \partial_{x_j} \dr_{x_k}  \bar v \\
+ \sum_{j,k = 1}^d  \Ll(\a_{ij}\Ll(\dr_{x_j} \phi_{e_k} \ast \Phi_{\lambda^{-1}}\Rr)   + \dr_{x_j} \S_{e,ij} \ast \Phi_{\lambda^{-1}}\Rr)\dr_{x_k} \bar v.
\end{multline}
\end{lemma}
\begin{proof}
The argument is very similar to that for \cite[Lemma~6.6]{AKMbook}, the main difference being that the definition of $\phil$ is slightly different from that of $\phi^\eps_{e_k}$ there. We recall the argument here for the reader's convenience. 
Observe that, for each $j \in \{1,\ldots,d\}$,
\begin{equation}  
\label{e.split.in.three}
\partial_{x_j} w = 
 \sum_{k = 1}^d \Ll( \Ll(\delta_{jk} + \partial_{x_j} \phi_{e_k}\Rr)\partial_{x_k} \bar v -\Ll(  \partial_{x_j} \phi_{e_k} \ast \Phi_{\lambda^{-1}}\Rr)\partial_{x_k} \bar v   +  \phil \, \partial_{x_j}\partial_{x_k} \bar v \Rr).
\end{equation}
We start by studying the contribution of the first summand. By~\eqref{e.def.S} and~\eqref{e.Sl}, we have, for every $i,k \in \{1,\ldots, d\}$,
\begin{equation*}  
\sum_{j = 1}^d \partial_{x_j} \Sl_{e_k,ij} = \sum_{j = 1}^d \Ll( \a_{ij} \Ll( \delta_{jk} + \partial_{x_j} \phi_{e_k} \Rr) - \ahom_{ij} \delta_{jk} -\partial_{x_j} \S_{e_k,ij} \ast \Phi_{\lambda^{-1}}\Rr) .
\end{equation*}
We deduce that, for each $i \in \{1,\ldots,d\}$,
\begin{equation}  
\label{e.integr.to.S}
\sum_{j,k = 1}^d \a_{ij}  \Ll(\delta_{jk} + \partial_{x_j} \phi_{e_k}\Rr) \partial_{x_k} \bar v
= \sum_{j,k = 1}^d \Ll( \ahom_{ij} \delta_{jk} + \partial_{x_j} \Sl_{e_k,ij} + \partial_{x_j} \S_{e_k,ij} \ast \Phi_{\lambda^{-1}}\Rr)  \partial_{x_k} \bar v,
\end{equation}
and thus
\begin{multline*}  
  \sum_{i,j,k = 1}^d \partial_{x_i} \Ll(\a_{ij}   \Ll(\delta_{jk} + \partial_{x_j} \phi_{e_k}\Rr)  \partial_{x_k} \bar v   \Rr) 
 =  \nabla \cdot \ahom \nabla  \bar v\\
 + \sum_{i,j,k = 1}^d \partial_{x_i} \Ll(  \partial_{x_j} \Sl_{e_k,ij} \, \partial_{x_k} \bar v \Rr) 
 + \sum_{i,j,k = 1}^d \dr_{x_i} \Ll( \Ll(\dr_{x_j} \S_{e_k,ij} \ast \Phi_{\lambda^{-1}} \Rr)\partial_{x_k} \bar v \Rr)  .
\end{multline*}
By the skew-symmetry of $\Sl_e$, we have
\begin{align*}  
0 & = \sum_{i,j,k = 1}^d \partial_{x_i} \partial_{x_j} \Ll( \Sl_{e_k,ij}\, \partial_{x_k} \bar v\Rr) \\
& = \sum_{i,j,k = 1}^d \partial_{x_i} \Ll( \partial_{x_j} \Sl_{e_k,ij} \, \partial_{x_k} \bar v\Rr)  + \sum_{i,j,k = 1}^d \partial_{x_i} \Ll(  \Sl_{e_k,ij} \, \partial_{x_j} \partial_{x_k} \bar v \Rr) ,
\end{align*}
and thus
\begin{multline*}  
  \sum_{i,j,k = 1}^d \partial_{x_i} \Ll(\a_{ij}   \Ll(\delta_{jk} + \partial_{x_j} \phi_{e_k}\Rr)  \partial_{x_k} \bar v   \Rr) 
 =  \nabla \cdot \ahom \nabla  \bar v\\
  - \sum_{i,j,k = 1}^d \partial_{x_i} \Ll(  \Sl_{e_k,ij} \, \partial_{x_j} \partial_{x_k} \bar v  \Rr) 
  +  \sum_{i,j,k = 1}^d \dr_{x_i} \Ll( \Ll(\dr_{x_j} \S_{e_k,ij} \ast \Phi_{\lambda^{-1}} \Rr) \partial_{x_k} \bar v\Rr)  .
\end{multline*}
Recalling \eqref{e.split.in.three}, we obtain the announced result.
\end{proof}

We next present the proof of Theorem~\ref{t.twoscale}, which can be compared to the one of~\cite[Theorem 6.9]{AKMbook}.

\begin{proof}[Proof of Theorem~\ref{t.twoscale}]
We will proceed by proving first \eqref{e.twoscale}, and then the $H^1$, $L^2$ and $H^{-1}$ estimates appearing in \eqref{e.error.estimates}, in this order. We decompose the arguments into seven steps.

\smallskip

\emph{Step 1.} We prove \eqref{e.twoscale}. In view of Lemma~\ref{l.two-scale}, it suffices to show that, for the vector field $\mathbf F$ defined in \eqref{e.def.Fi},
\begin{equation}
\label{e.est.F}
\|\mathbf F\|_{\un L^2(U_r)} \le \O_s  \Ll( C \ell(\lambda)\|\bar v\|_{\un H^{2}(U_r)} +  C \lambda^{\frac d 2} \, \|\bar v\|_{\un H^{1}(U_r)} \Rr).
\end{equation}
We estimate each of the terms appearing in the definition of $\mathbf F$. 
By Proposition~\ref{p.corrector} and \eqref{e.Osums}, we have, for every $i,j,k \in \{1,\ldots,d\}$,
\begin{equation*}
\Ll\| \Ll(\a_{ij} \phil -  \Sl_{e_k,ij}  \Rr)  \partial_{x_j} \dr_{x_k}  \bar v \Rr\|_{\un L^2(U_r)} \le \O_s \Ll( C \ell(\lambda) \|\bar v\|_{\un H^2(U_r)} \Rr) ,
\end{equation*}
as well as
\begin{equation*}  
\Ll\|\Ll(\a_{ij}\Ll(\dr_{x_j} \phi_{e_k} \ast \Phi_{\lambda^{-1}}\Rr)   + \dr_{x_j} \S_{e,ij} \ast \Phi_{\lambda^{-1}}\Rr)\dr_{x_k} \bar v\Rr\|_{\un L^2(U_r)} \le \O_s \Ll(C  \lambda^\frac d 2 \| \bar v\|_{\un H^{1}(U_r)} \Rr),
\end{equation*}
and thus \eqref{e.est.F} follows.

\smallskip

\emph{Step 2.}
In order to show \eqref{e.error.estimates}, we first need to evaluate the contribution of a boundary layer. 
For every $\ell \ge 0$, we write $\zeta_\ell := \ell^{-d} \zeta(\ell^{-1} \, \cdot\,)$ (recall the definition of $\zeta$ in \eqref{e.def.zeta}) and
\begin{equation}  
\label{e.def.url}
U_{r,\ell} := \Ll\{x \in U_r \ : \ \dist(x,\partial U_r) > \ell\Rr\}.
\end{equation}
With the definition of $\ell(\lambda)$ given in \eqref{e.def.llambda}, we set 
\begin{equation*}  
T := \Ll(\1_{\Rd \setminus U_{r,2\ell(\lambda)}}\ast \zeta_{\ell(\lambda)}\Rr) \sum_{k = 1}^d \phil \, \partial_{x_k} \bar v .
\end{equation*}
We will use the function $T$ as a test function for an upper bound on the size of the actual boundary layer in the next step. In this step, we show that there exists $C(s,U,\Lambda,\alpha,d) < \infty$ such that
\begin{equation}
\label{e.estim.nablaT}
\|\nabla T \|_{\un L^2(U_r)} \le \O_s \Ll( C \,\ell(\lambda)^\frac 1 2 \,\| \bar v \|_{\un H^1(U_r)}^\frac 1 2 \, \|\bar v\|_{\un H^2(U_r)}^\frac 1 2 + C \ell(\lambda) \|\bar v\|_{\un H^2(U_r)} \Rr) 
\end{equation}
and
\begin{equation}
\label{e.estim.T}
\|T\|_{\un L^2(U_r)} \le \O_s \Ll( C \, \ell(\lambda)^\frac 3 2  \,\| \bar v \|_{\un H^1(U_r)}^\frac 1 2 \, \|\bar v\|_{\un H^2(U_r)}^\frac 1 2  \Rr) .
\end{equation}

By the chain rule, 
\begin{equation*}  
\|\nabla T\|_{L^2(U_r)} 
\le C \sum_{k = 1}^d \Ll\| \Ll(\frac{ |\nabla \bar v| }{\ell(\lambda)} + |\nabla^2\bar v| \Rr) \, \Ll|\phil\Rr| + \Ll| \nabla \bar v  \Rr| \, \Ll|\nabla \phil\Rr| \Rr\|_{L^2(U_r\setminus U_{r,3\ell(\lambda)})}. 
\end{equation*}
By Proposition~\ref{p.corrector} and \eqref{e.Osums}, we have
\begin{equation*}  
\Ll\| |\nabla^2\bar v |  \, \Ll| \phil \Rr|  \Rr\|_{L^2(U_r \setminus U_{r,3\ell(\lambda)})} 
\le \O_s \Ll( C \ell(\lambda) \|\nabla^2 \bar v\|_{L^2(U_r)}  \Rr) .
\end{equation*}
Similarly,
\begin{equation}  
\label{e.for.estim.T}
\Ll\|  \frac{ |\nabla \bar v | }{\ell(\lambda)} \, \Ll| \phil \Rr| \Rr\|_{L^2(U_r \setminus U_{r,3\ell(\lambda)})} \le \O_s \Ll( C \|\nabla  \bar v\|_{L^2(U_r \setminus U_{r,3\ell(\lambda)})} \Rr) ,
\end{equation}
and by Proposition~\ref{p.trace},
\begin{equation}  
\label{e.est.boundary.barv}
\|\nabla  \bar v \|_{L^2(U_r \setminus U_{r,3\ell(\lambda)})} \le C \, \ell(\lambda)^\frac 1 2 \, r^{\frac d 2} \, \|\bar v\|_{\un H^1(U_r)}^\frac 1 2 \, \|\bar v\|_{\un H^2(U_r)}^\frac 1 2.
\end{equation}
Finally, using again Proposition~\ref{p.corrector} and \eqref{e.Osums}, we have
\begin{equation*}  
\Ll\| |\nabla \bar v | \, \Ll| \nabla \phil \Rr|  \Rr\|_{L^2(U_r \setminus U_{r,3\ell(\lambda)})} \le \O_s \Ll( C \|\nabla \bar v\|_{L^2(U_r \setminus U_{r,3\ell(\lambda)})} \Rr)  ,
\end{equation*}
and we can appeal once more to \eqref{e.est.boundary.barv} to estimate the norm of $\nabla \bar v$ on the right side above. This completes the proof of \eqref{e.estim.nablaT}. The estimate \eqref{e.estim.T} follows from \eqref{e.for.estim.T} and \eqref{e.est.boundary.barv}.

\smallskip

\emph{Step 3.}
We now evaluate the size of the boundary layer $b \in H^1(U_r)$ defined as the solution of
\begin{equation}  
\label{e.eq.b}
\Ll\{
\begin{aligned}  
& \Ll(\mu^2 - \nabla \cdot \a \nabla\Rr) b   = 0 & \quad \text{in } & U_r, \\
& b  = \sum_{k = 1}^d \phil \, \partial_{x_k} \bar v & \quad \text{on } & \partial U_r.
\end{aligned}
\Rr.
\end{equation}
Since $T$ and $b$ share the same boundary condition on $\partial U_r$, by the variational formulation of \eqref{e.eq.b}, we have
\begin{equation*}  
\int_{U_r} \Ll( \mu^2 b^2 + \nabla b \cdot \a \nabla b \Rr) \le \int_{U_r} \Ll( \mu^2 T^2 + \nabla T \cdot \a \nabla T \Rr).
\end{equation*}
By the result of the previous step, we thus obtain, for every $\mu \in [0,\lambda]$,
\begin{multline}  
\label{e.estim.b}
\mu \|b\|_{\un L^2(U_r)} + \|\nabla b\|_{\un L^2(U_r)} 
\\
\le \O_s \Ll( C \,\ell(\lambda)^\frac 1 2 \,\| \bar v \|_{\un H^1(U_r)}^\frac 1 2 \, \|\bar v\|_{\un H^2(U_r)}^\frac 1 2 + C \ell(\lambda) \|\bar v\|_{\un H^2(U_r)} \Rr) . 
\end{multline}

\smallskip

\emph{Step 4.}
We are now prepared to prove that
\begin{multline} 
\label{e.announce.X1}
\Ll\| \nabla (v - w) \Rr\|_{\un L^2(U_r)} + \mu  \|v-w\|_{\un L^2(U_r)} 
\\
\le  \O_s  \Ll( C  \Ll(\mu \ell(\lambda) + \lambda^\frac d 2\Rr) \|\bar v\|_{\un H^1(U_r)} +   C \ell(\lambda)^\frac 1 2 \,\| \bar v \|_{\un H^1(U_r)}^\frac 1 2 \, \|\bar v\|_{\un H^2(U_r)}^\frac 1 2 +  C \ell(\lambda) \|\bar v\|_{\un H^2(U_r)} \Rr).
\end{multline}
For concision, we define
\begin{equation*}  
\X_1 := \|-\nabla \cdot (\ahom \nabla \bar v - \a \nabla w)\|_{\un H^{-1}(U_r)},
\end{equation*}
and recall that, by \eqref{e.twoscale},
\begin{equation}
\label{e.estim.X1}
\X_1 \le \O_s \Ll( C \ell(\lambda) \|\bar v\|_{\un H^2(U_r)} + C \lambda^\frac d 2 \|\bar v\|_{\un H^1(U_r)} \Rr) .
\end{equation}
Moreover, by \eqref{e.eq.v.barv} and \eqref{e.eq.b},
\begin{align*}  
-\nabla \cdot (\ahom \nabla \bar v - \a \nabla w) 
&
= -\nabla \cdot \a \nabla (v-w) + \mu^2 (v-\bar v)  \\
& = -\nabla \cdot \a \nabla (v-w+b) + \mu^2(v-\bar v+b)   .
\end{align*}
Since $v-w+b \in H^1_0(U_r)$, we deduce that
\begin{multline*}  
|U_r|^{-1} \int_{U_r} \Ll(\nabla (v-w+b) \cdot \a \nabla (v-w+b) + \mu^2 (v-w+b)(v-\bar v+b) \Rr) 
\\
\le \X_1 \|\nabla(v-w+b)\|_{\un L^2(U_r)} ,
\end{multline*}
and by the uniform ellipticity of $\a$ and H\"older's inequality,
\begin{multline*}  
\|\nabla (v-w+b)\|_{\un L^2(U_r)}^2 + \mu^2 \|v-w+b\|_{\un L^2(U_r)}^2
\\
\le C \X_1 \|\nabla (v-w+b)\|_{\un L^2(U_r)} + \mu^2 \|w-\bar v\|_{\un L^2(U_r)} \, \|v-w+b\|_{\un L^2(U_r)}.
\end{multline*}
Using Proposition~\ref{p.corrector} and \eqref{e.Osums}, we verify that
\begin{equation}  
\label{e.est.w-barv}
\|w - \bar v\|_{\un L^2(U_r)} \le \O_s \Ll( C \ell(\lambda) \|\bar v\|_{\un H^1(U_r)} \Rr) .
\end{equation}
Combining these two estimates with \eqref{e.estim.X1} and Young's inequality, we obtain that
\begin{multline*}  
\|\nabla (v-w+b)\|_{\un L^2(U_r)} + \mu \|v-w+b\|_{\un L^2(U_r)} \\
\le \O_s \Ll(C \Ll(\mu \ell(\lambda) + \lambda^\frac d 2\Rr) \|\bar v\|_{\un H^1(U_r)}   + C \ell(\lambda) \|\bar v\|_{\un H^2(U_r)} \Rr) .
\end{multline*}
An application of \eqref{e.estim.b} then yields the announced estimate \eqref{e.announce.X1}.

\smallskip

\emph{Step 5.}
In this step, we complete the proof of the fact that 
$\|v-w\|_{\un H^1(U_r)}$
is bounded by the right side of \eqref{e.error.estimates}. In view of \eqref{e.announce.X1}, it suffices to show that
\begin{multline}  
\label{e.basic.l2}
r^{-1} \|v-w\|_{\un L^2(U_r)} \\
\le \O_s  \Ll( C \Ll(\mu \ell(\lambda) + \lambda^\frac d 2\Rr) \|\bar v\|_{\un H^1(U_r)} +  C \ell(\lambda)^\frac 1 2 \,\| \bar v \|_{\un H^1(U_r)}^\frac 1 2 \, \|\bar v\|_{\un H^2(U_r)}^\frac 1 2 +  C \ell(\lambda) \|\bar v\|_{\un H^2(U_r)} \Rr).
\end{multline}
By \eqref{e.estim.nablaT} and \eqref{e.announce.X1}, we have
\begin{multline*}  
\|\nabla (v-w+T)\|_{\un L^2(U_r)} \\
\le \O_s  \Ll( C \Ll(\mu \ell(\lambda) + \lambda^\frac d 2\Rr)\|\bar v\|_{\un H^1(U_r)} +  C \ell(\lambda)^\frac 1 2 \,\| \bar v \|_{\un H^1(U_r)}^\frac 1 2 \, \|\bar v\|_{\un H^2(U_r)}^\frac 1 2 +  C \ell(\lambda) \|\bar v\|_{\un H^2(U_r)} \Rr).
\end{multline*}
The estimate \eqref{e.basic.l2} then follows by the Poincar\'e inequality and \eqref{e.estim.T}. 

\smallskip

\emph{Step 6.}
We now complete the proof that $\Ll( \mu + r^{-1} \Rr) \|v-\bar v\|_{\un L^2(U_r)}$ is bounded by the right side of \eqref{e.error.estimates}. 
For $\mu \ge r^{-1}$, the result follows from \eqref{e.announce.X1} and \eqref{e.est.w-barv}, while $\mu \le r^{-1}$, it follows from \eqref{e.error.estimates} and \eqref{e.est.w-barv}. 

\smallskip

\emph{Step 7.}
We finally complete the proof of \eqref{e.error.estimates} by showing the estimate for the $H^{-1}$ norm of $v - \bar v$. 
If $\mu \le r^{-1}$, then the conclusion is immediate from the estimate on the $L^2$ norm of $v-\bar v$, by scaling. Otherwise, by the equations for $v$ and $\bar v$, we have 
\begin{equation*}  
\mu^2(v- \bar v) = \nabla \cdot \Ll( \a \nabla v - \ahom \nabla \bar v \Rr),
\end{equation*}
and moreover,
\begin{align*}  
& \Ll\| \nabla \cdot \Ll( \a \nabla v - \ahom \nabla \bar v\Rr)   \Rr\|_{\un H^{-1}(U_r)} 
\\
& \qquad \le 
\Ll\| \nabla \cdot \Ll( \a \nabla w - \ahom \nabla \bar v\Rr)   \Rr\|_{\un H^{-1}(U_r)} + \Ll\| \nabla \cdot \Ll( \a  \nabla v - \a \nabla w\Rr)   \Rr\|_{\un H^{-1}(U_r)} 
\\
& \qquad \le 
\Ll\| \nabla \cdot \Ll( \a \nabla w - \ahom \nabla \bar v\Rr)   \Rr\|_{\un H^{-1}(U_r)} + C \Ll\| \nabla v -  \nabla w \Rr\|_{\un L^2(U_r)}.
\end{align*}
The terms on the right side above have been estimated in \eqref{e.twoscale} and \eqref{e.announce.X1} respectively,  so the proof is complete.
\end{proof}

We next give the proof of the main result. 

\begin{proof}[Proof of Theorem~\ref{t.contract}]
Let $u,v,u_0,\bar u, \tilde u \in H^1(U_r)$ be as in the statement of Theorem~\ref{t.contract}. We first show the a priori estimates
\begin{equation}  
\label{e.apriori.u0}
\lambda \|u_0\|_{\un L^2(U_r)} + \|\nabla u_0\|_{\un L^2(U_r)} \le C \|u-v\|_{\un H^1(U_r)} ,
\end{equation}
and
\begin{equation}
\label{e.apriori.baru}
 \|\bar u\|_{\un H^1(U_r)} + \lambda^{-1} \|\bar u\|_{\un H^2(U_r)} \le C \|u-v\|_{\un H^1(U_r)} .
\end{equation}
By the variational formulation of the equation for $u_0\in H^1_0(U_r)$, we have
\begin{equation*}  
\int_{U_r} \Ll( \lambda^2 u_0^2 + \nabla u_0 \cdot \a \nabla u_0 \Rr) = \int_{U_r} \nabla u_0 \cdot \a \nabla (u-v).
\end{equation*}
By H\"older's and Young's inequalities and the uniform ellipticity of $\a$, we get~\eqref{e.apriori.u0}.
Using the equation \eqref{e.alt.eq.baru} satisfied by $\bar u \in H^1_0(U_r)$ and the estimate \eqref{e.apriori.u0}, we deduce
\begin{align*}  
\|\nabla \bar u\|_{\un L^2(U_r)} & \le C \|\nabla (u-v-u_0)\|_{\un L^2(U_r)} \\
& \le C \|\nabla (u-v)\|_{\un L^2(U_r)}.
\end{align*}
By Proposition~\ref{p.h2.estim} and the $L^2$ estimate in \eqref{e.apriori.u0},
we also have
\begin{equation*}  
\|\bar u\|_{\un H^2(U_r)} \le C \lambda^2 \|u_0\|_{\un L^2(U_r)} \le C \lambda \|u-v\|_{\un H^1(U_r)},
\end{equation*}
as announced in \eqref{e.apriori.baru}.

\smallskip

We now introduce the two-scale expansion
\begin{equation*}  
w := \bar u  + \sum_{k = 1}^d \phil \, \partial_{x_k} \bar u.
\end{equation*}
Using the equation for $\bar u$ in \eqref{e.alt.eq.baru} and Theorem~\ref{t.twoscale} with $\mu = 0$, we obtain
\begin{multline*}  
\|v + u_0 + w - u\|_{\un H^1(U_r)} \\
\le \O_s \Ll( C \lambda^{\frac d 2} \|\bar u\|_{\un H^1(U_r)} + C \ell(\lambda)^\frac 1 2 \|\bar u\|_{\un H^1(U_r)}^\frac 1 2 \, \|\bar u\|_{\un H^2(U_r)}^\frac 1 2 +  C \ell(\lambda) \|\bar u\|_{\un H^2(U_r)}  \Rr) ,
\end{multline*}
and thus, by \eqref{e.apriori.baru},
\begin{equation}  
\label{e.h-1.expansion}
\|v + u_0  + w - u  \|_{\un H^1(U_r)} \le \O_s \Ll( C \ell(\lambda)^\frac 1 2 \lambda^\frac 1 2 \|u-v\|_{\un L^2(U_r)} \Rr) .
\end{equation}
In order to complete the proof of Theorem~\ref{t.contract}, there remains to estimate the $H^1$ norm of $w- \tilde u$. By the equation for $\tilde u$, Theorem~\ref{t.twoscale} and \eqref{e.apriori.baru}, we have
\begin{align*}  
& \Ll\|  w - \tilde u  \Rr\|_{\un H^1(U_r)} 
\\
& \qquad \le \O_s  \bigg( C \Ll(\lambda \ell(\lambda) + \lambda^\frac d 2\Rr) \|\bar u\|_{\un H^1(U_r)} \\
& \qquad \qquad \qquad \qquad \qquad +  C \ell(\lambda)^\frac 1 2 \,\| \bar u \|_{\un H^1(U_r)}^\frac 1 2 \, \|\bar u\|_{\un H^2(U_r)}^\frac 1 2 +  C \ell(\lambda) \|\bar u\|_{\un H^2(U_r)} \bigg) 
\\
& \qquad \le \O_s \Ll( C \ell(\lambda)^\frac 1 2 \lambda^\frac 1 2 \|u-v\|_{\un H^1(U_r)} \Rr) ,
\end{align*}
as desired.
\end{proof}

%
%
%
%
%
%
\section{Numerical results}
\label{s.numerics}

In this section, we report on numerical tests demonstrating the performance of the iterative method described in Theorem~\ref{t.contract}. The code used in the tests can be consulted at
\begin{equation}
\label{e.github}
\mbox{\small{\url{https://github.com/ahannuka/homo_mg}}}
\end{equation}
Throughout this section, we consider a two-dimensional {\em random checkerboard} coefficient field $x \mapsto \a(x)$, which is defined as follows: we give ourselves a family $(b(z))_{z \in \Z^2}$ of independent random variables such that for every $z \in \Z^2$,
\begin{equation*}  
\P \Ll[ b(z) = 1 \Rr]  = \P \Ll[ b(z) = 9 \Rr]  = \frac 1 2.
\end{equation*}
We then set, for every $x \in z + [0,1)^2$,%
\begin{equation*}
\a(x) := b(z) \, I_2,
\end{equation*}
where $I_2$ denotes the $2$-by-$2$ identity matrix. For this particular coefficient field, the homogenized matrix can be computed analytically as $\ahom = 3 I_2$ (see \cite[Exercise~2.3]{AKMbook}). When such an analytical expression does not exist, the homogenized coefficient can be approximated numerically, for example, by using the method presented in \cite{efficient}.
\begin{figure}[ht]
\includegraphics[scale=0.7]{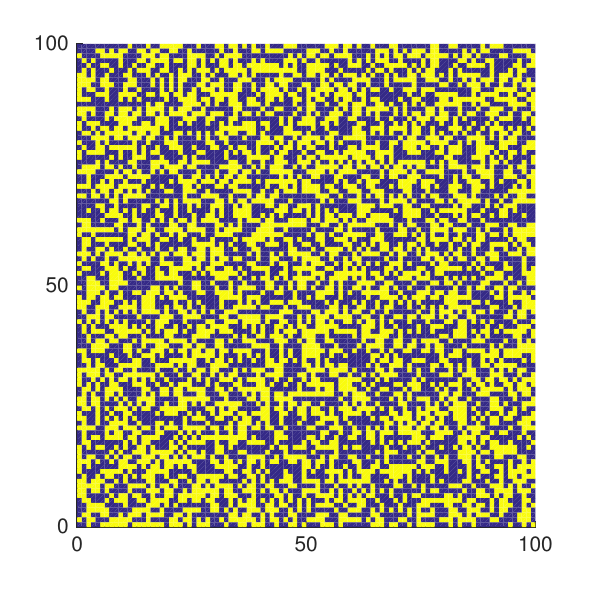}
\includegraphics[scale=0.35]{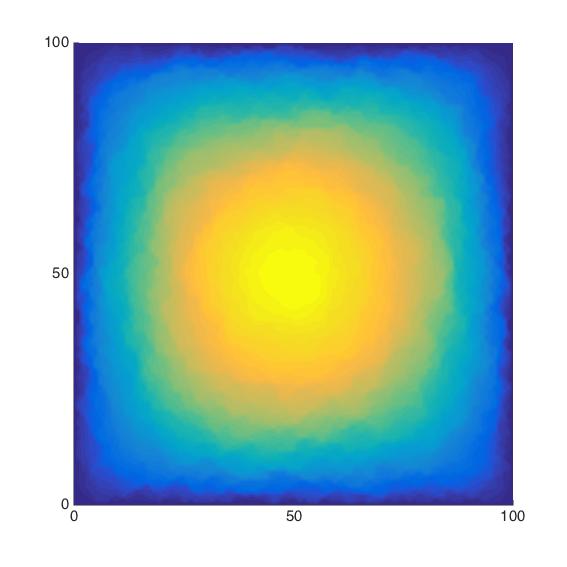}
\caption{On the left, a typical realization of the coefficient field $\a(x)$, with $r=100$ (yellow coresponds to the value $1$ and blue to the value $9$). On the right, the corresponding solution.}
\label{fig:sol_cof}
\end{figure}
\begin{figure}[ht]
\begin{center}
\includegraphics[scale=0.7]{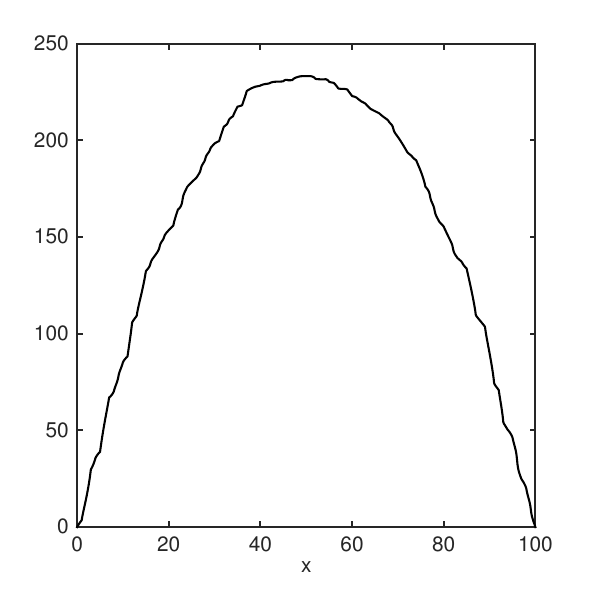}
\includegraphics[scale=0.7]{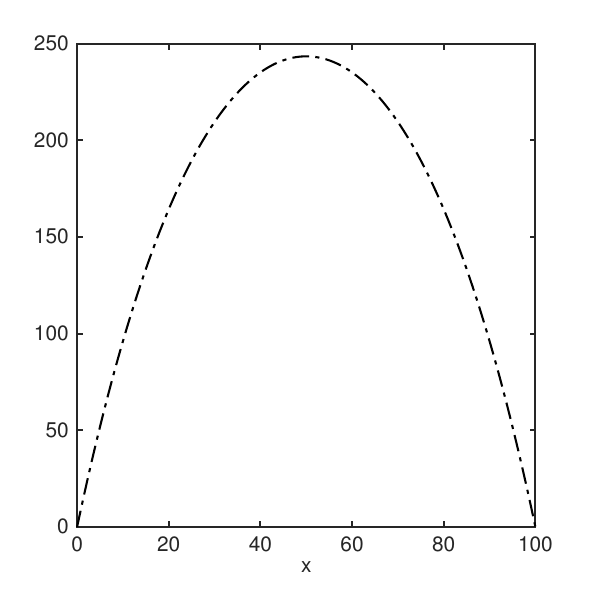}
\end{center}
\caption{On the left, the FE-solution to the heterogeneous problem, and on the right, the FE-solution to the corresponding homogenized problem. Both solutions are plotted along the line $y=55$. The fast oscillation in the left figure is clearly visible.}
\label{fig:sol_y55}
\end{figure}

For each $r > 0$, we write $U_r := (0,r)^2$. We aim to compute the solution to the continuous partial differential equation in \eqref{e.model.problem} with $w = 0$ (null Dirichlet boundary condition) and load function $f = 1$. We discretize this problem using a first-order finite element method. Let  $\mathcal{T}$ be a triangular mesh of the domain $U_r$ constructed by first dividing each cell~$z + [0,1)^2$ ($z \in \Z^2$) into two triangles, and then using three levels of uniform mesh refinement. This results into a sufficiently fine mesh to capture the oscillations present in the exact solution $u$.  The first order finite element space 
\begin{equation*}
V_h := \Ll\{ u \in H^1_0(U_r) \; | \; u_{|K} \in P^1(K) \; \quad \forall K \in \mathcal{T}  \Rr\}
\end{equation*}
with standard nodal basis is used in all computations. The finite element solution $u_h \in V_h$ satisfies 
\begin{equation}
\label{eq:fem}
\forall v_h \in V_h, \quad ( \a \nabla u_h, \nabla v_h) = (f,v_h) . 
\end{equation}

\begin{figure}[ht]
\begin{center}
\includegraphics[scale=0.7]{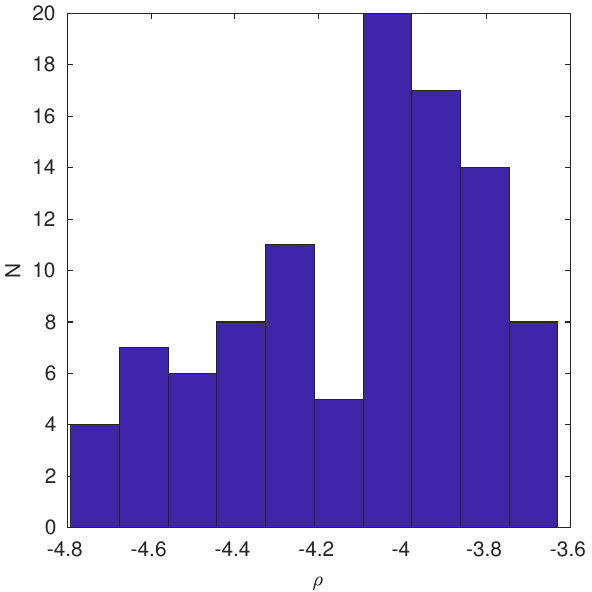} 
\includegraphics[scale=0.7]{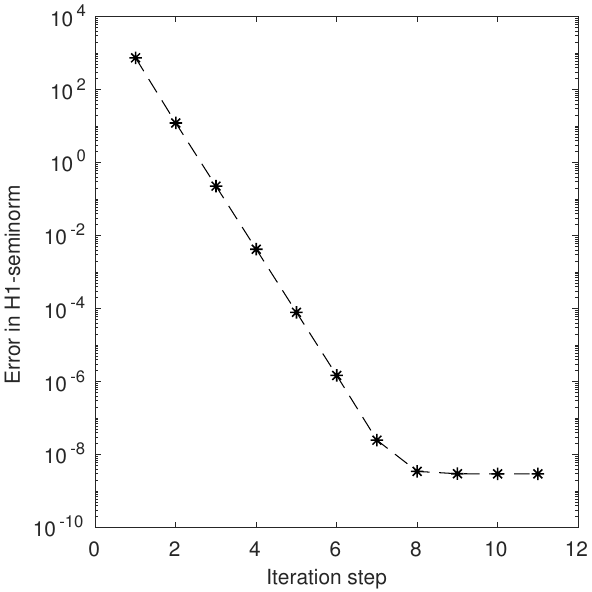}
\end{center}
\caption{On the left, the empirical distribution of the factor $\rho$ for~$\lambda = 0.1$ and $r=100$, based on $100$ runs. On the right, the error in the $H^1$ seminorm for $r=100$ and~$\lambda = 0.1$, after each iteration. The method converges after $8$ iterations.} 
\label{fig:cdist}
\end{figure}

A typical realization of the coefficient field $\a(x)$ and of the corresponding exact solution $u$ are visualized in Figure \ref{fig:sol_cof}, with the choice of $r = 100$. The high-frequency oscillations in the solution are clearly visible in Figure \ref{fig:sol_y55}, where the solution is visualized along the line $y=55$. 

\smallskip

Our interest lies in the contraction factor of the iterative procedure. The contraction factor is studied by first solving the finite dimensional problem (\ref{eq:fem}) exactly using a direct solver. Then a sequence of approximate solutions $\{ u_h^{(i)} \}_{i=1}^N$ is generated by starting from $u_h^{(1)} = 0$ and applying the iterative procedure described in Theorem~\ref{t.contract}. The logarithm of the error $\|\nabla( u - u_h^{(i)}) \|_{L^2(U_r)}$ is computed for each $i \in \{1,\ldots,10\}$, a regression line is fitted, and the slope of this line is denoted by~$\rho$. It is our numerical estimate of the logarithm of the contraction factor; roughly speaking,
\begin{equation*}
\rho \approx \log{ \left( \frac{  \| \nabla (u_h - u_h^{(i+1)} )\|_{L^2(U_r)}}{ \| \nabla (u_h - u_h^{(i)} ) \|_{L^2(U_r)}} \right)}
\end{equation*}
(``$\log$'' denotes the natural logarithm.)
The iteration is said to converge when the relative error is smaller than $10^{-9}$. Past this threshold, the error between the exact and the iterative solutions is smaller than the accuracy of the discretization itself, and thus cannot be measured. 

\smallskip

Since the coefficient field is random, the contraction factor will vary for different realizations of $\a$. For the choice of $\lambda = 0.1$ and $r = 100$, the empirical distribution of the contraction factor is given in Figure \ref{fig:cdist}, based on one hundered samples of the coefficient field. Appart from the purposes of displaying this histogram, each of our estimates for $\rho$ is an average over ten realizations of the coefficient field. 

\smallskip

In our first test, the parameter $\lambda$~is fixed to $\lambda = 0.1,0.2$, and then $0.4$. The size of the domain $r$ is varied between $10$ and $200$. The averaged contraction factor is visualized on the left side of Figure \ref{fig:rho_fig}. The results are in excellent agreement with Theorem~\ref{t.contract}. After a pre-asymptotic region, the contraction factor becomes independent of the size of the domain $r$. The pre-asymptotic region is due to the fact that for small values of $r$, the pre- and post-smoothing steps are essentially sufficient to solve the equation. We emphasize that the contraction factor remains very good, of the order of $0.1$, even for the relatively large value of $\lambda = 0.4$. 

\smallskip

In the second test, the size of the domain $r$ takes values $r = 100$, $200$, and~$300$, while $\lambda$~is varied between~$0.01$~and~$0.5$. For each~$\lambda$, the exponent of the averaged contraction factor is computed based on ten simulation runs. The results are presented on the right side of Figure \ref{fig:rho_fig}. After a pre-asymptotic region, the exponential of the contraction factor behaves like $\lambda^{1/2}$, as predicted by Theorem~\ref{t.contract}. The pre-asymptotic region is roughly characterized by the scaling $r \lesssim 10 \lambda^{-1}$.
\begin{figure}[ht]
\begin{center}
\includegraphics[scale=0.7]{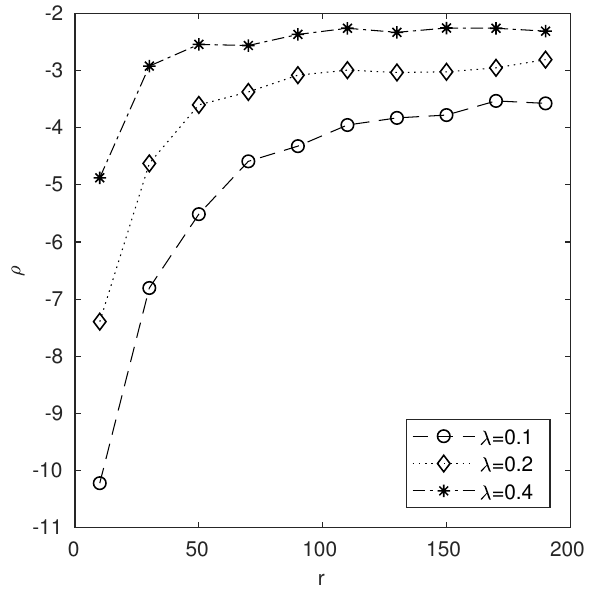}\hfill
\includegraphics[scale=0.7]{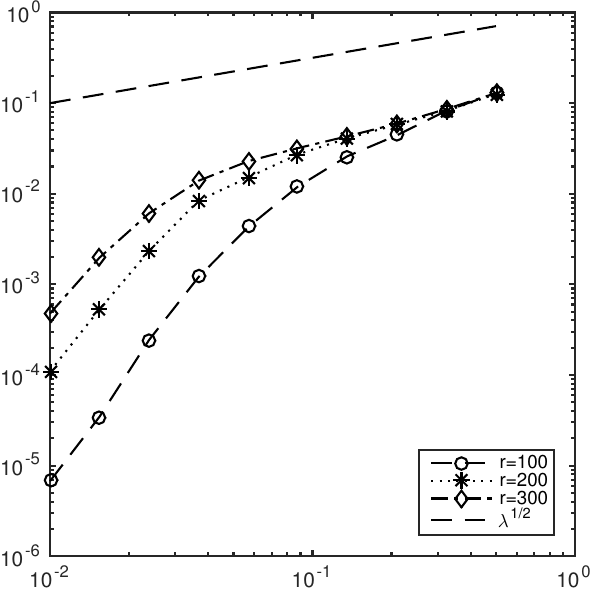}
\end{center}
\caption{On the left, averaged factor $\rho$ as a function of $r$, for~$\lambda = 0.1$, $0.2$, and~$0.4$. On the right, the exponential of the averaged factor $\rho$ as a function of $\lambda$ for $r=100$, $200$, and $300$. In all cases, the average is computed from ten simulation runs.}
\label{fig:rho_fig}
\end{figure}

%

%
%
%
%
%
%

\appendix

\section{Sobolev estimates}
\label{s.append}

In this appendix, we prove an estimate for the norm of a function restricted to a layer close to the boundary of a domain. The estimate is an integrated version of a trace theorem. For convenience, we will also recall a standard $H^2$ estimate for homogeneous elliptic equations. As in \eqref{e.def.url}, for every $\ell \ge 0$, we write
\begin{equation*}  
U_{r,\ell} := \Ll\{x \in U_r \ : \ \dist(x,\partial U_r) > \ell\Rr\}.
\end{equation*}
\begin{proposition}[Trace estimate]
\label{p.trace}
There exists $C(U,d) < \infty$ such that for every $r \ge 1$, $\ell \in \Ll(0,r\Rr]$ and $f \in H^1 (U_r)$,
\begin{equation*}  
r^{-d} \, \|f\|_{L^2(U_r \setminus U_{r,\ell})}^2 \le C \, \ell  \, \|f\|_{\un L^2(U_r)} \, \|f \|_{\un H^1(U_r)}.
\end{equation*}
\end{proposition}
\begin{proof}
Denote by $\mathbf n_{r,t}$ the unit normal vector to $\partial U_{r,\ell}$, which we extend to $U_{r,\ell}$ harmonic continuation. Since $U$ is $C^{1,1}$, there exists $C(U,d) < \infty$ such that for every $t \in (0,r/C]$, we have $\|\nabla \mathbf n_{r,t}\|_{L^\infty(U_{r,t})} \le Cr^{-1}$. It thus follows that
\begin{align*}  
\int_{\partial U_{r,t}} f^2 
&
= \int_{U_{r,t}} \nabla \cdot (f^2 \, \mathbf{n}_{r,t}) 
\\
& 
\le C r^{-1} \|f\|^2_{L^2(U_r)} + C \|f\|_{L^2(U_r)} \|\nabla f \|_{L^2(U_r)}
\\
&
\le C r^d \|f\|_{\un L^2(U_r)} \, \|f\|_{\un H^1(U_r)}.
\end{align*}
By the coarea formula, for every $\ell \le r /C$, we have
\begin{equation*}  
\|f\|_{L^2(U_r \setminus U_{r,\ell})}^2 = \int_0^\ell \|f\|^2_{L^2(\partial U_{r,t})} \, dt.
\end{equation*}
Combining the previous two displays yields
\begin{equation*}  
\|f\|_{L^2(U_r \setminus U_{r,\ell})}^2 \le C \ell \, r^d \, \|f\|_{\un L^2(U_r)} \, \|f\|_{\un H^1(U_r)},
\end{equation*}
which is the announced result. The case $\ell > r/C$ is immediate.
\end{proof}
\begin{proposition}[$H^2$ estimate] 
\label{p.h2.estim}
Let $\ahom \in \R^{d\times d}_{\mathrm{sym}}$ satisfy \eqref{e.ahom.ell}. There exists a constant $C(\Lambda, U,d) < \infty$ such that for every $u \in H^1_0(U_r)$ and $f \in L^2(U_r)$, if 
\begin{equation*}  
-\nabla \cdot \ahom \nabla u = f,
\end{equation*}
then $u \in H^2(U_r)$ and
\begin{equation*}  
\| u\|_{\un H^2(U_r)} \le C \|f\|_{\un L^2(U_r)}.
\end{equation*}
\end{proposition}
\begin{proof}
See \cite[Theorem~6.3.2.4]{Evans}.
\end{proof}

\subsection*{Acknowledgments} SA~was partially supported by the NSF Grant DMS-1700329. AH~was partially supported by the Stenb\"{a}ck stiftelse.  TK~was supported by the Academy of Finland. JCM was partially supported by the ANR Grant LSD (ANR-15-CE40-0020-03).

\small
\bibliographystyle{abbrv}
\bibliography{fixedpoint}

\end{document}